\documentclass[a4paper,11pt]{article}
\usepackage{CJKutf8}
\usepackage{CJK}
\usepackage{amsfonts}
\usepackage{pifont}
\usepackage{bm}
\usepackage{latexsym,amsmath,amssymb,cite,amsthm}
\usepackage{color,eucal,enumerate,mathrsfs}
\usepackage[normalem]{ulem}
\usepackage{amsmath}
\usepackage{verbatim}
\usepackage{hyperref}

\numberwithin{equation}{section}
\theoremstyle{plain}
\newtheorem{theorem}{Theorem}[section]
\newtheorem{corollary}[theorem]{Corollary}
\newtheorem{lemma}[theorem]{Lemma}
\newtheorem{proposition}[theorem]{Proposition}
\theoremstyle{definition}
\newtheorem{definition}[theorem]{Definition}
\theoremstyle{remark}

\newtheorem{remark}[theorem]{Remark}

\newcommand{\Opt}{\mathrm{OptGeo}}

\newcommand{\geo}{\rm Geo}

\newcommand{\brac}[1]{\left(#1\right)}
\newcommand{\lmts}[2]{\mathop{\overline{\lim}}_{{#1} \rightarrow {#2}} }

\newcommand{\MCP}{{\rm MCP}}
\newcommand{\CD}{{\rm CD}}
\newcommand{\mm}{\mathfrak m}
\newcommand{\ms}{(X,\d,\mm)}

\newcommand{\cdkn}{{\rm CD}(K, N)}

\newcommand{\bekn}{{\rm BE}(K, N)}
\newcommand{\mcp}{{\rm MCP}(K, N)}
\newcommand{\vol}{{\rm Vol}_{\rm g}}
\newcommand{\g}{{\rm g}}
\newcommand{\qq}{\mathfrak{q}}

\newcommand{\R}{\mathbb{R}}

\renewcommand{\H}{\mathbb{H}}

\DeclareMathOperator*{\esssup}{ess\,sup}
\newcommand{\osc}{\mathop{\rm osc}\nolimits} 
\newcommand{\diam}{\mathop{\rm diam}\nolimits} 
\newcommand{\supp}{\mathop{\rm supp}\nolimits}   
\newcommand{\conv}{\mathop{\rm geo}\nolimits}   
   
\newcommand{\Lip}{\mathop{\rm Lip}\nolimits}

\renewcommand{\d}{{\mathrm d}}

\newcommand{\restr}[1]{\lower3pt\hbox{$|_{#1}$}}

\newcommand{\eps}{\varepsilon}
\newcommand{\nchi}{{\raise.3ex\hbox{$\chi$}}}

\renewcommand*{\thefootnote}{\fnsymbol{footnote}}

\title{\large{\bf Sharp Poincar\'e inequalities under Measure Contraction Property}
}

\begin{document}
\author{Bang-Xian Han\textsuperscript{$*$,$\dagger$}
\and
Emanuel Milman\textsuperscript{$*$,$\ddagger$}
}

\footnotetext[1]{Department of Mathematics, Technion-Israel Institute of Technology, Haifa 32000, Israel.}
\footnotetext[2]{Email: hanbangxian@gmail.com.}
\footnotetext[3]{Email: emilman@tx.technion.ac.il.}

\begingroup    \renewcommand{\thefootnote}{}    
    \footnotetext{The research leading to these results is part of a project that has received funding from the European Research Council (ERC) under the European Union's Horizon 2020 research and innovation programme (grant agreement No 637851).}
\endgroup

\date{} 
\maketitle

\begin{abstract}
We prove a sharp Poincar\'e inequality for subsets $\Omega$ of (essentially non-branching) metric measure spaces satisfying the Measure Contraction Property $\MCP(K,N)$, whose diameter is bounded above by $D$. This is achieved by identifying the corresponding one-dimensional model densities and a localization argument, ensuring that the Poincar\'e constant we obtain is best possible as a function of $K$, $N$ and $D$. Another new feature of our work is that we do not need to assume that $\Omega$ is geodesically convex, by employing the geodesic hull of $\Omega$ on the energy side of the Poincar\'e inequality. In particular, our results apply to geodesic balls in ideal sub-Riemannian manifolds, such as the Heisenberg group. 
\end{abstract}

\section{Introduction}

Determining the optimal constant in the Poincar\'e inequality, which in an appropriate setting is equivalent to the spectral-gap of a corresponding Laplacian, is one of the most classical problems in comparison geometry. Given a metric measure space $(X,\d,\mm)$, its associated Poincar\'e constant is given by
\[ \lambda_{(X,\d,\mm)}:= \inf \left \{\frac {\int_{X}  |{\nabla }_{ X} f|^2 \,\mm}{\int_{X} |f|^2  \,\mm}: f \in \Lip_{loc}(X,\d) , \int_{X} f  \,\mm =0,  0 < \int_X |f|^2 \mm < \infty \right \},
 \]  where $\Lip_{loc}(X,\d)$ denotes the class of locally Lipschitz functions, and the local Lipschitz constant  $ |{\nabla }_{X} f|: X \mapsto \R$ is defined as
 \[
  |{\nabla }_{X} f|(x):=\lmts{y}{x} \frac{|f(y)-f(x)|}{\d(y, x)} 
\]
(and $0$ if $x$ is an isolated point). Under very general assumptions on $(X,\d,\mm)$, it is known \cite{AGS-CalculusAndHeatFlow,AGS-DensityOfLipFuncs} that Lipschitz functions are dense in the Sobolev space $W^{1,2}(X,\d,\mm)$, and hence the above definition may be equivalently stated using Sobolev functions; as a matter of convenience, we employ Lipschitz functions throughout this work. 
Given a family $\mathcal F:=\{(X_\alpha, \d_\alpha, \mm_\alpha): \alpha \in \mathcal A\}$  of metric measure spaces, we define the optimal Poincar\'e constant $\lambda_{\mathcal F}$ on $\mathcal F$ by:
\[
\lambda_{\mathcal F}:= \mathop{\inf}_{\alpha \in \mathcal A} \lambda_{(X_\alpha,\d_\alpha,\mm_\alpha)} . 
\]

One of the most studied families of metric measure spaces are smooth connected orientable compact Riemannian manifolds $(M,\g,\vol)$ with Ricci curvature bounded below by $K\in \R$, dimension bounded above by $N\in [1,\infty]$, and diameter bounded above by $D\in (0,\infty]$; the manifolds are typically allowed to have locally convex boundary (in the sense that the second fundamental form on $\partial M$ is positive semi-definite). In this case, $\lambda_{(M,\g,\vol)}$ is the first positive eigenvalue of the Laplace-Beltrami operator $-\Delta_\g$ with vanishing Neumann boundary conditions. Two well-known examples are:
\begin{itemize}
\item The Lichnerowicz theorem \cite{LichnerowiczBook} (see also \cite{EscobarLichnerowiczWithConvexBoundary,XiaLichnerowiczWithConvexBoundary} for the case when $\partial M \neq \emptyset$ is locally convex) asserts that $\lambda_{\mathcal F} = \frac{N}{N-1} K$ for the family $\mathcal F$ of $N$-dimensional manifolds as above when $K > 0$ and $D=\infty$. 
\item The  Li--Yau \cite{LiYauEigenvalues} and Zhong--Yang \cite{ZhongYangImprovingLiYau} theorems assert that $\lambda_{\mathcal F} = \frac{\pi^2}{D^2}$ for the family $\mathcal F$ of manifolds as above when $K=0$. \end{itemize}

More generally, one may equip $M$ with a measure $\mu$ having smooth positive density with respect to $\vol$, thereby obtaining the family of weighted Riemannian manifolds $(M,\g,\mu)$. In this setting, $\lambda_{(M,\g,\mu)}$ coincides with the first positive eigenvalue of an appropriate weighted Laplacian $-\Delta_{\g,\mu}$, and the Bakry--\'Emery Curvature-Dimension condition $\bekn$  gives rise to natural generalized notions of Ricci curvature lower bound $K$ and dimension upper bound $N$ \cite{B-H, BE-D}. 
Based on a refined gradient comparison technique originating in the work of Kr\"oger \cite{Kroger-GradientComparison} and a careful analysis of the underlying model spaces, sharp estimates on $\lambda_{\mathcal F}$ for the family of weighted Riemannian manifolds satisfying $\bekn$ and whose diameter is bounded above by $D$ were obtained by Bakry and Qian in \cite{BakryQianSharpSpectralGapEstimatesUnderCDD}, and the associated one-dimensional model spaces were identified. In particular, the Lichnerowicz and Li--Yau / Zhong-Yang theorems continue to hold under $\bekn$. 

Thanks to the development of Optimal Transport theory, it was realized that the Bakry--\'Emery condition $\bekn$ in the smooth setting can  be equivalently characterized by $K$-convexity of an $N$-entropy functional along $L^2$-Wasserstein geodesics  \cite{CMS01,SVR-T}. Motivated by this, an appropriate  $\cdkn$ Curvature-Dimension condition for (possibly non-smooth) metric measure spaces was introduced independently by Sturm \cite{SturmCD1, SturmCD2} and Lott--Villani \cite{Lott-Villani09,LottVillani-WeakCurvature}, encapsulating a certain synthetic Ricci curvature lower bound $K$ and dimension upper bound $N$. The class of metric measure spaces satisfying $\cdkn$, which includes all of the previous smooth examples, has the advantage of being closed in the measured Gromov-Hausdorff topology. Naturally, analyzing $\lambda_{\mathcal F}$ in this generality presents a greater challenge, since many of the analytic tools from the smooth setting are not available any longer. 

Fortunately, one tool which is nowadays available is the localization technique. In the Euclidean setting, this method has its roots in the work of Payne and Weinberger \cite{PayneWeinberger} on the spectral-gap for convex domains in Euclidean space, and has been further developed by Gromov and V.~Milman \cite{Gromov-Milman} and Kannan, Lov\'asz and Simonovits \cite{KLS}. Roughly speaking, the localization paradigm reduces the task of establishing various analytic and geometric inequalities on an $n$-dimensional space to a one-dimensional problem. Recently, in a ground-breaking work \cite{KlartagLocalizationOnManifolds}, B.~Klartag reinterpreted the localization paradigm as a measure disintegration adapted to $L^1$-Optimal-Transport, and extended it to weighted Riemannian manifolds satisfying $\bekn$. In a subsequent breakthrough, Cavalletti and Mondino \cite{CavallettiMondino-Localization} have succeeded to extend this technique to a very general subclass of metric measure spaces satisfying $\cdkn$. Using the localization method, in conjunction with an extremal point characterization of one-dimensional measures satisfying $\bekn$ and a careful analysis of the spectral-gap for one-dimensional model measures, the Bakry--Qian sharp estimates on the Poincar\'e constant have been extended in E.~Calderon's Ph.D. thesis \cite{Calderon18} to the entire range $N \in (-\infty,0] \cup [2,\infty]$ (his results are formulated for smooth weighted manifolds satisfying the $\bekn$ condition, but apply equally well to non-smooth $\cdkn$ spaces on which the localization technique is available). 

\medskip

Another property of metric measure spaces was introduced independently by Ohta \cite{Ohta-MCP} and Sturm \cite{SturmCD2} as a weaker variant of the $\cdkn$ condition. This property, called the Measure Contraction Property $\mcp$, is equivalent to the $\cdkn$ condition on smooth unweighted $N$-dimensional Riemannian manifolds, 
but may be strictly weaker for more general spaces. It was shown by Juillet \cite{Juillet09} that the $n$-dimensional Heisenberg group $\H^n$, which is the simplest example of a non-trivial sub-Riemannian manifold,  equipped with the Carnot-Carath\'eodory metric and (left-invariant) Lebesgue measure, does not satisfy the $\cdkn$ condition for \emph{any} $K,N \in \R$, but does satisfy ${\rm MCP}(0,N)$  for $N = 2n+3$.
More general Carnot groups were shown to satisfy $\MCP(0,N)$ for appropriate $N$ by Barilari and Rizzi \cite{Rizzi-MCPonCarnot,BarilariRizzi-MCPonHTypeCarnot}. 
Very recently, interpolation inequalities \`a la Cordero-Erausquin--McCann--Schmuckenshl\"ager \cite{CMS01} have been obtained, under suitable modifications, by Balogh, Krist\'aly and Sipos  \cite{BKS18} for the Heisenberg group and by Barilari and Rizzi \cite{BarilariRizzi18} in the general ideal sub-Riemannian setting.  As a consequence, additional examples of spaces verifying ${\rm MCP}$ but not ${\rm CD}$ have been found, e.g. generalized H-type groups, the Grushin plane and Sasakian structures
(under appropriate curvature lower bounds; for more details, see \cite{BarilariRizzi18}).

Fortunately, the localization paradigm still applies to very general $\mcp$ spaces; this observation has its roots in the work of Bianchini and Cavalletti in the non-branching setting \cite{BianchiniCavalletti-MongeProblem}, and was extended to essentially non-branching $\mcp$ spaces with $N < \infty$ in \cite{Cavalletti-MongeForRCD,CavallettiEMilman-LocalToGlobal,CM-Laplacian}. It is known from the work of Figalli and Rifford \cite{FigalliRifford10} that ideal sub-Riemannian manifolds are indeed essentially non-branching (see Section \ref{sec:prelim} for precise definitions). 
Using localization, Cavalletti and Santarcangelo \cite{CS18} have recently obtained sharp isoperimetric inequalities on $\mcp$ spaces having diameter upper bounded by $D$, by identifying an appropriate family of one-dimensional model $\mcp$ densities. Their work extends the work of the second named author on smooth $\cdkn$ spaces \cite{Milman-JEMS}, where an appropriate family of one-dimensional model $\cdkn$ measures was identified, and which was subsequently generalized to the non-smooth setting by Cavalletti and Mondino \cite{CavallettiMondino-Localization}.

\medskip

In this work, we study the Poincar\'e  inequality in the class $\mathcal{MCP}_{K,N,D}$ of essentially non-branching metric measure spaces verifying the Measure Contraction Property $\mcp$ (with $K \in \R$, $N \in (1,\infty)$) and having diameter upper bounded by $D \in (0,\infty)$. As in the $\cdkn$ setting, determining the sharp constants in analytic one-dimensional inequalities is \emph{a-priori} more difficult than their isoperimetric counterparts, and in particular, we do not know how to obtain the extremal point characterization of one-dimensional $\mcp$ measures (as in \cite{Calderon18} for $\cdkn$ measures). Fortunately, we are able to identify the ``worst" $\mcp$ density supported on an interval of diameter $D$ for the spectral-gap problem by a direct ODE comparison argument, thereby determining the optimal constant $\lambda_{{\mathcal {MCP}}_{K,N,D}}$. 

An additional feature of this work is that we formulate our results on the Poincar\'e inequality for general subsets $\Omega$ of a $\mcp$ space $(X,\d,\mm)$, with $\diam(\Omega) \leq D$.  This is very important for applications, since the $\mcp$ condition forces $(\supp(\mm),\d)$ to be a geodesic space, and so whenever $\Omega \subset \supp(\mm)$ is not geodesically convex, $(\Omega,\d,\mm|_\Omega)$ does not satisfy $\mcp$, and hence results proven for $\mcp$ spaces are not directly applicable. On the other hand, geodesically convex subsets of sub-Riemannian spaces are particularly scarce -- for instance, even for the simplest case of the Heisenberg group $\H^1$, it was shown in \cite{MontiRickly} that the smallest geodesically convex set containing three distinct points which do not lie on a common geodesic is $\H^1$ itself, implying in particular that there are no non-trivial geodesically convex balls in $\H^1$. 

The idea which permits us to handle a general domain $\Omega$ is new even in the $\CD(K,N)$ setting, and immediately allows to extend the sharp Poincar\'e inequalities from \cite{BakryQianSharpSpectralGapEstimatesUnderCDD,Calderon18} (or any other Sobolev inequality) from geodesically convex domains to general domains, in the manner described next. Given a subset $\Omega \subset \supp(\mm)$, denote by $\conv(\Omega)$ its geodesic hull, namely the union of all geodesics starting at $x \in \Omega$ and ending at $y \in \Omega$. Note that $\conv(\Omega)$ need not be geodesically convex, and that $\conv(B_r(x_0)) \subset B_{2r}(x_0)$ by the triangle inequality for any geodesic ball $B_r(x_0)$ of radius $r > 0$. The idea is to use $\conv(\Omega)$ on the energy side of the Poincar\'e inequality. 

Our main result thus reads as follows. Abbreviate $\lambda[h] = \lambda_{(\R, |\cdot| , h \mathcal{L}^1)}$ for the Poincar\'e constant of the density $h$ (with respect to the Lebesgue measure $\mathcal{L}^1$ on $\R$). For $\kappa \in \R$, define the function $s_\kappa: [0, +\infty) \mapsto \R$ (on $[0, \pi/ \sqrt{\kappa})$ if $\kappa >0$) as:
\begin{equation} \label{eq:skappa}
s_\kappa(\theta):=\left\{\begin{array}{lll}
(1/\sqrt {\kappa}) \sin (\sqrt \kappa \theta), &\text{if}~~ \kappa>0,\\
\theta, &\text{if}~~\kappa=0,\\
(1/\sqrt {-\kappa}) \sinh (\sqrt {-\kappa} \theta), &\text{if} ~~\kappa<0.
\end{array}
\right.
\end{equation}
Denote by $D_{K,N}$ the Bonnet--Myers diameter upper-bound:
\begin{equation} \label{eq:DKN}
D_{K,N} := \left\{\begin{array}{lll} \frac{\pi}{\sqrt{K / (N-1)}} & \text{if}~~ K > 0 \\
+\infty & \text{otherwise} 
\end{array} \right .  .
\end{equation}
We refer to Section \ref{sec:prelim} for other missing definitions. 

\begin{theorem} \label{thm:main}
Let $(X,\d,\mm)$ denote an essentially non-branching metric measure space satisfying $\MCP(K,N)$, with $K \in \R$ and $N \in (1,\infty)$. Let $\Omega \subset \supp(\mm)$ be a closed subset 
with $\diam(\Omega) \leq D < \infty$. Then for any (locally) Lipschitz function $f : (X,\d) \rightarrow \R$,
\[
\int_{\Omega} f \mm = 0  \;\; \Rightarrow \;\; \lambda_{{\mathcal {MCP}}_{K,N,D}} \int_{\Omega} f^2 \mm \leq \int_{\conv(\Omega)} |\nabla_{X} f|^2 \mm ,
\]
where:
\begin{equation} \label{eq:intro-optimum}
\lambda_{{\mathcal {MCP}}_{K,N,D}} := \left\{\begin{array}{lll}  
\lambda[h_{K,N,D}] & \text{if}~~ K \leq 0 \\
\inf_{D' \in (0,\min(D,D_{K,N})]} \lambda[h_{K,N,D'}] & \text{if}~~ K > 0 \\
\end{array}
\right . ,
\end{equation}
and $h_{K,N,D}$ denotes the following one-dimensional $\MCP(K,N)$ density:
\[
h_{K,N,D} (x) := \left\{\begin{array}{lll}  
s^{N-1}_{K/(N-1)}(D-x) & \text{if}~~ x \in [0,D/2]  \\
s^{N-1}_{K/(N-1)}(x) & \text{if}~~ x \in [D/2,D] \\
\end{array}
\right .    .
\]
\end{theorem}

\medskip

The above Poincar\'e constant is clearly best possible for the class of subsets $\Omega$ of $\MCP(K,N)$ spaces with $\diam(\Omega) \leq D$, as witnessed by the one-dimensional $\MCP(K,N)$ spaces $([0,D'],|\cdot|,h_{K,N,D'} \mathcal{L}^1)$ and $\Omega = [0,D']$ (with $D'=D$ when $K \leq 0$ and $D' \in (0,\min(D,D_{K,N})]$ when $K  >0$). The difference between the cases $K \leq 0$ and $K > 0$ was already observed in \cite{CS18} in the isoperimetric context; in Section \ref{sec:estimates}, we demonstrate that this is \textbf{not} an artifact of the proof, but rather a consequence of the fact that $(0,D_{K,N}] \ni D' \mapsto \lambda[h_{K,N,D'}]$ is \textbf{not} monotone non-increasing when $K > 0$. We also obtain various concrete estimates on $\lambda_{{\mathcal {MCP}}_{K,N,D}}$; in particular:
\[
\lambda_{{\mathcal {MCP}}_{K,N,D}} \geq \left\{\begin{array}{lll}  
 \frac{1}{4} \frac{N^2}{D^2} \frac{1}{2^{N-1}} & \text{if}~~ K \geq 0 \\
 \frac{1}{4} \max\brac{|K|(N-1), \frac{N^2}{D^2}} \brac{\frac{\sinh(\sqrt{\frac{-K}{N-1}} \frac{D}{2})}{\sinh (\sqrt{\frac{-K}{N-1}} D)}}^{N-1} 
   & \text{if}~~ K < 0 \\
\end{array}
\right . .
\]

\medskip

We stress again that by the results of \cite{Juillet09,FigalliRifford10,Rizzi-MCPonCarnot,BarilariRizzi-MCPonHTypeCarnot,BarilariRizzi18}, Theorem \ref{thm:main} applies to the ideal sub-Riemannian setting. We illustrate this here for the simplest example of geodesic balls in the Heisenberg group $\H^n$ (equipped with the Carnot-Carath\'eodory metric $\d$ and Lebesgue measure $\mathcal{L}^{2n+1}$, which satisfies $\MCP(0,2n+3)$ by \cite{Juillet09}):
\begin{corollary}
For any Lipschitz function $f : (\H^n,\d) \rightarrow \R$, $x_0 \in \H^n$ and $r > 0$:
\[
\int_{B_r(x_0)} f \mathcal{L}^{2n+1} = 0  \;\; \Rightarrow \;\;  \frac{1}{4} \frac{(2n+3)^2}{(2r)^2} \frac{1}{2^{2n+2}} \int_{B_r(x_0)} f^2  \mathcal{L}^{2n+1}\leq \int_{B_{2r}(x_0)} |\nabla_{\H^n} f|^2 \mathcal{L}^{2n+1}  .
\]
\end{corollary}
\noindent While the validity of a local Poincar\'e inequality on $\H^n$ is well-known (even in tight form, where $B_{2r}$ on the right-hand-side is replaced by $B_r$), starting from the work of D.~Jerison on vector fields satisfying H\"{o}rmander's condition \cite{Jerison-HormanderCondition} (see also \cite{DongLuSun} and the references therein), we are not aware of  any explicit constants in these inequalities. Note that by \cite{SobolevMetPoincare}, it is always possible to tighten a Poincar\'e inequality on any geodesic space, but this would result in somewhat of a loss of explicit constants. 

\bigskip
The rest of this work is organized as follows. In Section \ref{sec:prelim} we recall some preliminaries on metric measure spaces. In Section \ref{sec:ODE}, we derive our basic ODE comparison principle. In Section \ref{sec:1D}, we apply the comparison principle to one-dimensional $\MCP(K,N)$ densities and identify the extremal model densities $h_{K,N,D}$. In Section \ref{sec:estimates}, we derive various estimates on $\lambda[h_{K,N,D}]$ as a function of the parameters $K$, $N$ and $D$. In Section \ref{sec:proof}, we prove Theorem \ref{thm:main}. In Section \ref{sec:conclude}, we compare to some previously known results pertaining to Poincar\'e inequalities on $\MCP(K,N)$ spaces due to Sturm, von Renesse and others. In a subsequent work \cite{Han-MCP-pPoinare}, the results of this work will be extended to $p$-Poincar\'e inequalities along with corresponding rigidity results for cases of equality.

\bigskip
\noindent \textbf{Acknowledgments}.  We thank the referee for carefully reading the manuscript and providing helpful comments. 

\section{Preliminaries on Metric Measure Spaces} \label{sec:prelim}

Let $(X, \d)$ be a complete separable metric space endowed with a locally finite Borel measure $\mm$ -- 
such triplets $(X,\d,\mm)$ will be called metric measure spaces. We refer to \cite{AG-U,AGS-G,Gromov,VillaniTopicsInOptimalTransport, VillaniOldAndNew} for background on metric measure spaces in general, and the theory of optimal transport on such spaces in particular. 

We denote by $\geo(X,\d)$ the set of all closed directed constant-speed geodesics parametrized on the interval $[0,1]$. 
We regard $\geo(X,\d)$ as a subset of all Lipschitz maps $\text{Lip}([0,1], X)$ endowed with the uniform topology. 
Recall that $(X,\d)$ is called a \emph{geodesic} metric space (or simply geodesic) if for any $x,y \in X$ there exists $\gamma \in \geo(X,\d)$ with $\gamma_0 = x$ and $\gamma_1 = y$. 
Given a subset $A$ of a geodesic space $(X,\d)$, we denote by $\conv(A)$ the geodesic hull of $A$, namely:
\[
\conv(A) := \cup_{\{ \gamma \in \geo(X,\d) \; ; \; \gamma_0,\gamma_1 \in A \}}  \gamma \; ;
\]
note that $\conv(A)$ need not be a geodesic space itself. 

The space of all Borel probability measures on $(X,\d)$ is denoted by $\mathcal{P}(X)$. It is naturally equipped with its weak topology, in duality with bounded continuous functions $C_b(X)$ over $X$.  
The subspace of those measures having finite second moment will be denoted by $\mathcal{P}_{2}(X)$. 
The weak topology on $\mathcal{P}_{2}(X)$ is metrized by the $L^{2}$-Wasserstein distance $W_{2}$, defined as follows for any $\mu_0,\mu_1 \in \mathcal{P}(X)$:
\begin{equation}\label{eq:Wdef}
  W_2^2(\mu_0,\mu_1) := \inf_{ \pi} \int_{X\times X} \d^2(x,y) \, \pi(dx , dy),
\end{equation}
where the infimum is taken over all $\pi \in \mathcal{P}(X \times X)$ having $\mu_0$ and $\mu_1$ as the first and the second marginals, respectively.
It is known that the infimum in (\ref{eq:Wdef}) is always attained for any $\mu_0,\mu_1 \in \mathcal{P}(X)$. When $\mu_0,\mu_1 \in \mathcal{P}_2(X)$ then 
this minimum is necessarily finite, and a transference plan realizing it is called an optimal transference plan between $\mu_0$ and $\mu_1$. 

As $(X,\d)$ is a complete and separable metric space then so is $(\mathcal{P}_2(X), W_2)$. 
Under these assumptions, it is known that $(X,\d)$ is geodesic if and only if $(\mathcal{P}_2(X), W_2)$ is geodesic. Given $t \in [0,1]$, let $e_t$ denote the evaluation map: 
\[
  e_{t} : \geo(X,\d) \ni \gamma\mapsto \gamma_t \in X.
\]
A measure $\Pi \in  \mathcal{P}(\geo(X,\d))$ is called an optimal dynamical plan if $(e_0,e_1)_{\sharp} \Pi$ is an optimal transference plan; it easily follows in that case that
$[0,1] \ni t \mapsto (e_t)_\sharp \nu$ is a geodesic in $(\mathcal{P}_2(X), W_2)$. It is known that any geodesic $(\mu_t)_{t \in [0,1]}$ in $(\mathcal{P}_2(X), W_2)$ can be lifted to an optimal dynamical plan $\Pi$ so that $(e_t)_\sharp \Pi  = \mu_t$ for all $t \in [0,1]$ (c.f. \cite[Theorem 2.10]{AG-U}). We denote by $\Opt(\mu_{0},\mu_{1})$ the space of all optimal dynamical plans $\Pi$ so that $(e_i)_\sharp \Pi = \mu_i$, $i=0,1$.

\subsection{Essentially Non-Branching Spaces}

We say that a subset $\Gamma \subset  \geo(X,\d)$  is non-branching  if  for any $\gamma^1, \gamma^2 \in \Gamma$, it holds:
\[
\exists t \in (0,1) ~~\text{s.t.}~ ~\forall  s\in [0, t]~ \gamma_s^1 =\gamma_s^2 \Rightarrow  \forall s \in [0,1] \;\; \gamma_s^1 =\gamma_s^2 . 
\]
We say that $\ms$ is  essentially non-branching \cite{RS-N} if for any $\mu_0, \mu_1 \ll \mm$ in $\mathcal{P}_2(X)$, any $\Pi \in \Opt(\mu_{0},\mu_{1})$  is  concentrated on a Borel non-branching subset of geodesics. 

\smallskip
The restriction to essentially non-branching spaces is natural and facilitates 
avoiding pathological cases: as an example of possible pathological behaviour we mention the
failure of the local-to-global property of $\cdkn$ within this class of spaces; in particular,
a heavily-branching metric measure space verifying a local version of $\CD(0, 4)$ which does not verify $\cdkn$ for any
fixed $K \in \R$ and $N \in [1,\infty]$ was constructed by Rajala in \cite{RajalaFailureOfLocalToGlobal}, while the local-to-global
property of $\cdkn$ has been recently verified in \cite{CavallettiEMilman-LocalToGlobal} for essentially non-branching metric measure spaces. 

\smallskip

It is clear that if $(X, \d)$ is a smooth complete Riemannian manifold $(M,\g)$ (with its induced geodesic distance) then $\geo(X,\d)$ is non-branching, and so in particular $(M,\g,\mm)$ is essentially non-branching for any measure $\mm$. In addition, very general complete sub-Riemannian manifolds $(M,\Delta,\g)$
equipped with their volume measure $\mm$ are also  essentially non-branching  (see Figalli and Rifford \cite[Section 4]{FigalliRifford10}), as follows from the existence and uniqueness of the optimal transport map on such spaces \cite[Theorem 3.3 and Section 3.4]{FigalliRifford10}; for instance, this holds for all ideal sub-Riemannian structures,
that is admitting no non-trivial abnormal minimizing geodesics \cite[Theorem 5.9]{FigalliRifford10}.

\subsection{$\MCP(K,N)$}

As already mentioned in the Introduction, the Measure Contraction Property $\MCP(K,N)$ was introduced by Ohta \cite{Ohta-MCP} and Sturm \cite{SturmCD2} as a weaker variant of the $\cdkn$ condition. On general metric measure spaces the two definitions slightly differ, but on essentially non-branching spaces they coincide, and so we use the simplest definition to state. 

Recall the definition of the function $s_\kappa$ from (\ref{eq:skappa}), and the Bonnet-Myers upper bound $D_{K,N}$ from (\ref{eq:DKN}). 
Given $K \in \R$ and $N \in (1,\infty)$, we set for $(t, \theta) \in [0, 1] \times \R^+$, 
\begin{equation} \label{eq:sigma}
 \sigma^{(t)}_{K, N-1}   \big(\theta):=\left\{\begin{array}{llll}
+\infty &\text{if}~~ \theta \geq D_{K,N}, \\ \frac{s_{K/(N-1)}(t \theta)}{s_{K/(N-1)}(\theta)} &\text{otherwise},\\
\end{array}
\right.
\end{equation}
and
\[
 \tau^{(t)}_{K, N}(\theta) :=t^{\frac 1N} \Big ( \sigma^{(t)}_{K, N-1}(\theta)     \Big )^{1-\frac1N}.
\]

\begin{definition}[Measure contraction property $\mcp$]
A metric measure space $(X,\d,\mm)$ is said to satisfy $\MCP(K,N)$ if for any $o \in \supp(\mm)$ and  $\mu_0 \in \mathcal{P}_2(X)$ of the form $\mu_0 = \frac{1}{\mm(A)} \mm\llcorner_{A}$ for some Borel set $A \subset X$ 
with $0 < \mm(A) < \infty$,
there exists $\Pi \in \Opt(\mu_0, \delta_{o} )$ such that:
\begin{equation} \label{eq-MCP}
\frac{1}{\mm(A)} \mm \geq (e_{t})_{\sharp} \big( \tau_{K,N}^{(1-t)}(\d(\gamma_{0},\gamma_{1}))^{N} \Pi(d \gamma) \big) \;\;\; \forall t \in [0,1] .
\end{equation}
\end{definition}

From \cite[Proposition 9.1 $(i) \Leftrightarrow (iv)$]{CavallettiEMilman-LocalToGlobal}, an equivalent definition is to require the existence of $\Pi \in \Opt(\mu_0, \delta_{o})$ so that $\mu_t := (e_t)_{\#} \Pi \ll \mm$ for all $t \in [0,1)$, and that writing $\mu_t = \rho_t \mm$, we have for all $t \in [0,1)$:
\[ 
\frac{1}{\rho_t(\gamma_t)} \geq \tau_{K,N}^{(1-t)}(\d(\gamma_0,\gamma_1))^N \mm(A) \;\;\; \text{for $\Pi$-a.e. $\gamma \in \geo(X,\d)$} .
\]
On an essentially non-branching space satisfying $\MCP(K,N)$, it follows from the results of \cite{CavallettiMondino-ENB-MCP} that the above $\Pi$ is unique and is induced by a map (i.e. $\Pi = S_{\sharp}(\mu_0)$ for some map $S: X \rightarrow \geo(X,\d)$). 

If $(X,\d,\mm)$ satisfies $\MCP(K,N)$ with $N \in (1,\infty)$ then $(\supp(\mm),\d)$ is proper and geodesic (e.g. \cite{Ohta-MCP,CavallettiEMilman-LocalToGlobal}). Furthermore, it was shown in \cite{Ohta-MCP,SturmCD2} that when $K > 0$, the following (sharp) Bonnet--Myers diameter bound holds:
\begin{equation} \label{eq:BM}
{\diam(\supp \mm)} \leq D_{K,N} ;
\end{equation}
we remark that while this is obvious from our present definition and the fact that $\tau_{K,N}(\theta) = +\infty$ if $\theta \geq D_{K,N}$, the above bound was shown in \cite{Ohta-MCP} under an a-priori weaker (but ultimately equivalent) definition of $\MCP(K,N)$ where $A$ is assumed to be a subset of $B(o, D_{K,N})$ and in addition $(\supp(\mm) , \d)$ is a-priori assumed to be a length-space.

\subsection{Localization}

Recall that given a measure space $(X,\mathscr{X},\mm)$, a set $A \subset X$ is called $\mm$-measurable if $A$ belongs to the completion of the $\sigma$-algebra $\mathscr{X}$, generated by adding to it all subsets of null $\mm$-sets; similarly, a function $f : (X,\mathscr{X},\mm) \rightarrow \R$ is called $\mm$-measurable if all of its sub-level sets are $\mm$-measurable. We denote by $\mathcal{M}(X,\mathscr{X})$ the collection of measures on $(X,\mathscr{X})$. $\mm$ is said to be concentrated on $A \subset X$ if $\exists B \subset A$ with $B \in \mathscr{X}$ so that $\mm(X \setminus B) = 0$.

\begin{definition}[Disintegration on sets] \label{def:disintegration}
\label{defi:dis}
Let $(X,\mathscr{X},\mm)$ denote a measure space. 
Given any family $\{X_q\}_{q \in Q}$ of subsets of $X$, a \emph{disintegration of $\mm$ on $\{X_q\}_{q \in Q}$} is a measure-space structure 
$(Q,\mathscr{Q},\qq)$ and a map
\[
Q \ni q \longmapsto \mm_{q} \in \mathcal{M}(X,\mathscr{X})
\]
so that:
\begin{itemize}
\item For $\qq$-a.e. $q \in Q$, $\mm_q$ is concentrated on $X_q$.
\item For all $B \in \mathscr{X}$, the map $q \mapsto \mm_{q}(B)$ is $\qq$-measurable.
\item For all $B \in \mathscr{X}$, $\mm(B) = \int_Q \mm_{q}(B)\, \qq(dq)$; this is abbreviated by  $\mm = \int_Q \mm_{q} \qq(dq)$.
\end{itemize}
\end{definition}

\begin{theorem}[Localization for $\mcp$ spaces]\label{thm:localization}
Let $\ms$ be an essentially non-branching metric measure space
satisfying the $\mcp$ condition for some $K\in \R$ and $N \in (1, \infty)$. Let $g : X \rightarrow \R$ be $\mm$-integrable with $\int_X g \mm = 0$ and $\int_X|g(x)| \d(x,x_0) \mm(dx) < \infty$ for some (equivalently, all) $x_0 \in X$. Then there exists an $\mm$-measurable subset $\mathsf T \subset X$ and a family $\{X_{q}\}_{q \in Q} \subset X$, such that: 
\begin{enumerate}
\item There exists a disintegration of $\mm\restr{\mathsf{T}}$ on $\{X_{q}\}_{q \in Q}$:
\[
\mm\restr{\mathsf{T}} = \int_{Q} \mm_{q} \, \qq(dq)  ~,~ \qq(Q) = 1 . 
\]
\item For $\qq$-a.e. $q \in Q$, $X_q$ is a closed geodesic in $(X,\d)$.
\item For $\qq$-a.e. $q \in Q$, $\mm_q$ is a Radon measure supported on $X_q$ with $\mm_q \ll  \mathcal H^1 \restr{X_q}$.
\item For $\qq$-a.e. $q \in Q$, the metric measure space $(X_{q}, \d,\mm_{q})$ verifies $\MCP(K,N)$.
\item For $\qq$-a.e. $q \in Q$, $\int g \mm_q = 0$, and $g \equiv 0$  $\mm$-a.e. on $X \setminus \mathsf T$. 
\end{enumerate}
\end{theorem}

The localization paradigm on $\mcp$ spaces has its roots in the work of Bianchini and Cavalletti in the non-branching setting (c.f. \cite[Theorem 9.5]{BianchiniCavalletti-MongeProblem}), and was extended to essentially non-branching $\mcp$ spaces with $N < \infty$ and finite $\mm$ in \cite[Theorem 7.10 and Remark 9.2]{CavallettiEMilman-LocalToGlobal} (building upon \cite{Cavalletti-MongeForRCD}) and for general $\mm$ in \cite[Theorem 3.5]{CM-Laplacian}. The idea to use $L^1$-transport between the positive and negative parts $g_+ := \max(g,0)$ and $g_- := (-g)_+$ of the balanced function $g$ to ensure that it remains balanced along the localization is due to Klartag \cite{KlartagLocalizationOnManifolds} (see \cite{CavallettiMondino-Localization} for an adaptation to the metric measure space setting). 

\begin{proof}[Proof of Theorem \ref{thm:localization}]
Simply combine \cite[Theorem 3.5]{CM-Laplacian} with the proof of \cite[Theorem 5.1]{CavallettiMondino-Localization}. Up to modification on a $\mm$-null-set, the set $T$ is the transport set of the $1$-Lipschitz Kantorovich potential $u$ associated to the $L^1$-Optimal-Transport between $g_+ \mm$ and $g_- \mm$, which consists of geodesics $\{X_q\}$ on which the function $u$ is affine with slope $1$. 
 \end{proof}

\section{ODE Comparison Principle} \label{sec:ODE}

It is well known and easy to see (see Section \ref{sec:proof}) that the Localization Theorem reduces the study of the Poincar\'e constant on metric measure spaces satisfying $\MCP(K,N)$ to the one-dimensional case. To understand the one-dimensional setting, we observe in this section a simple comparison principle for ODEs. We refer to \cite{Zettl-SLP} for well-known facts from classical Sturm--Liouville theory. 

\medskip

Given a compact interval $I  = [a,b]\subset \R$, consider the density $\Psi = \exp(V)$ where $V$ is a continuous piecewise smooth function on $I$. 
Denote the weighted Laplacian $\Delta_\Psi$ acting on $f \in C^{\infty}(I)$ by:
\[
\Delta_\Psi f := f'' + V' f' . 
\]
Let $C_{*,\dagger}^{\infty}(I)$ denote the subset of $C^{\infty}(I)$ consisting of functions satisfying $*$-boundary condition at $a$ and $\dagger$-boundary condition at $b$; here $*,\dagger \in \{D , N\}$, $D$ stands for zero Dirichlet boundary condition and $N$ stands for zero Neumann boundary condition. It is well known that as an operator on $L^2(I,\Psi)$, $-\Delta_\Psi$ with domain $C_{*,\dagger}^{\infty}(I)$ is essentially self-adjoint and positive semi-definite for any $*,\dagger \in \{D , N\}$. Denoting the corresponding self-adjoint extension by $-\Delta^{*,\dagger}_\Psi$, it is also well-known that $-\Delta^{*,\dagger}_\Psi$ has discrete spectrum, consisting of an increasing sequence of simple (multiplicity one) non-negative eigenvalues $\{\Lambda_i\}$ tending to $+\infty$. The associated eigenfunctions $\{ u_i \}$ and their first derivatives are absolutely continuous on $I$, and are smooth in any open subset of $I$ where $\Psi$ is; they satisfy the corresponding boundary conditions and $u_i'' + V' u_i' = -\lambda_i u_i$ in the distributional sense on $I$. 

We denote by $\Lambda^{*,\dagger}(\Psi,I)$ the first non-zero eigenvalue of $-\Delta^{*,\dagger}_\Psi$; for $\{*,\dagger\} \in \{ \{N,N\} , \{D,N\} \}$, the associated eigenfunction is strictly monotone on $I$, and in particular, has a single zero in the interior of $I$ when $\{*,\dagger\} =  \{N,N\}$. 
By domain monotonicity (which for $\{*,\dagger\}  \neq\{ D,D\}$ is a particular feature of the one-dimensional setting), $\Lambda^{*,\dagger}(\Psi , [\xi,\eta])$ is a continuous function of $(\xi,\eta)$ in $a \leq \xi < \eta \leq b$, strictly decreasing as $\eta \nearrow$ or as $\xi \searrow$. 

\begin{lemma}[ODE Comparison Principle] \label{lem:ODE}
Assume $0 \in (a,b)$. Let $V_0$ be continuous on $[a,b]$, and smooth on $[a,0]$ and on $[0,b]$. Denote $\Psi_0 = \exp(V_0)$, and assume that the eigenfunction $u_0$ of $-\Delta^{N,N}_{\Psi_0}$ associated to $\Lambda^{N,N}(\Psi_0,[a,b])$ has its (unique) zero at $0$. 

Then for all $V \in C^{\infty}([a,b])$, if $V' \geq V_0'$ on $[a,0)$ and $V' \leq V_0'$ on $(0,b]$, then denoting $\Psi = \exp(V)$ we have:
\[
\Lambda^{N,N}(\Psi,[a,b]) \geq \Lambda^{N,N}(\Psi_0,[a,b]) . 
\]
\end{lemma}
\begin{proof}
Let $u$ denote the eigenfunction of $-\Delta^{N,N}_{\Psi}$ associated to the first non-zero eigenvalue $\Lambda^{N,N}(\Psi,[a,b])$, and let $\xi \in (a,b)$ denote its (unique) zero. Clearly:
\[
\phantom{\lambda_0 :=} \Lambda^{N,N}(\Psi,[a,b]) = \Lambda^{N,D}(\Psi,[a,\xi]) = \Lambda^{D,N}(\Psi,[\xi,b]) ,
\]
and:
\[
\lambda_0 := \Lambda^{N,N}(\Psi_0,[a,b]) = \Lambda^{N,D}(\Psi_0,[a,0]) = \Lambda^{D,N}(\Psi_0,[0,b]) .
\]

Assume first that $\xi \in [0,b]$. We will show that $\Lambda^{D,N}(\Psi,[\xi,b]) \geq \Lambda^{D,N}(\Psi_0,[0,b]) = \lambda_0$, thereby establishing the assertion. 
If this were not the case, then by domain monotonicity we would have $\Lambda^{D,N}(\Psi,[\xi,\eta]) = \lambda_0$ for some $\eta \in (\xi,b)$. Let $u \in C^{\infty}([\xi,\eta])$ be the corresponding monotone increasing (non-negative) eigenfunction solving:
\[
-\Delta_{\Psi} u = \lambda_0 u ~,~ u(\xi) = 0 ~,~ u'(\eta) = 0 . 
\]
Since $u' \geq 0$ and $V' \leq V_0'$ on $[\xi,\eta] \subset [0,b]$, this implies that:
\[
-\Delta_{\Psi_0} u \leq \lambda_0 u  \text{ on } [\xi,\eta] . 
\]
Using the non-negativity of $u$, we deduce that:
\[
\int_{\xi}^{\eta} u (-\Delta_{\Psi_0} u) \Psi_0 dx \leq \lambda_0 \int_{\xi}^{\eta} u^2 \Psi_0 dx , 
\]
and so by the min-max theorem, 
we conclude that $\Lambda^{D,N}(\Psi_0,[\xi,\eta]) \leq \lambda_0$. On the other hand, domain monotonicity implies that $\Lambda^{D,N}(\Psi_0,[\xi,\eta]) > \Lambda^{D,N}(\Psi_0,[0,b]) = \lambda_0$, and we obtain our desired contradiction. 

If $\xi \in [a,0]$, we conclude by a similar argument that $\Lambda^{N,D}(\Psi,[a,\xi]) \geq \Lambda^{N,D}(\Psi_0,[a,\xi]) = \lambda_0$ (now $u$ is the non-positive monotone increasing eigenfunction corresponding to $\Lambda^{N,D}(\Psi,[\eta,\xi]) = \lambda_0$, and $-\Delta_{\Psi_0} u \geq \lambda_0 u$ on $[\eta,\xi]$). 
\end{proof}

\section{One dimensional model} \label{sec:1D}

\subsection{One dimensional $\MCP$ densities}

We say that a non-negative $h \in L^1_{loc}(\R,\mathcal L^1)$  is a $\mcp$ density if:
\begin{equation}\label{eq:mcp}
h(tx_1+(1-t)x_0) \geq \sigma^{(1-t)}_{K, N-1} (|x_1-x_0|)^{N-1} h(x_0)
\end{equation}
for all $x_0, x_1 \in \supp h$ and $t\in [0,1]$. We use $\supp h$ throughout to denote $\supp(h \mathcal L^1)$, where, recall, $\mathcal L^1$ denotes the Lebesgue measure on $\R$.
The following is well-known:

\begin{lemma} \label{lem:MCP-density}
The one-dimensional metric-measure space $(\R, |\cdot|,  h \mathcal L^1)$ satisfies $\mcp$ if and only if (up to modification on a null-set) $h$ is a $\mcp$ density. 
\end{lemma}
\begin{proof}
The if direction follows from \cite[Corollary 5.5 (i)]{SturmCD2}. The only if direction follows by considering the $\mcp$ condition for uniform measures $\mu_0,\mu_1$ on intervals of length $\eps$ and $\alpha \eps$, respectively, letting $\eps \rightarrow 0$, employing Lebesgue's differentiation theorem, and optimizing on $\alpha > 0$ (e.g. as in the proof of \cite[Theorem 4.3]{Cavalletti-Sturm12} or \cite[Theorem 3.3.6]{Calderon18}).
\end{proof}

\begin{definition}
Given $K \in \R$, $D \in (0,\infty)$ and $N \in (1,\infty)$, we define $ \mathcal {MCP}^1_{K, N, D}$ as the collection of $\mcp$ densities $h\in L^1(\R, \mathcal L^1)$ with $\supp h=[0, D]$.
\end{definition}

Recalling the definitions of $\sigma_{K, N-1}$ and $s_{\kappa}$ from (\ref{eq:sigma}) and (\ref{eq:skappa}), it is immediate to check that \eqref{eq:mcp}  is equivalent to the requirement that ${\diam(\supp h)} \leq D_{K,N}$ and:
\begin{equation}\label{eq1}
\left (\frac {s_{K/(N-1)}(b-x_1)}{s_{K/(N-1)}(b-x_0)} \right )^{N-1} \leq \frac{h(x_1)}{h(x_0)} \leq \left (\frac {s_{K/(N-1)}(x_1-a)}{s_{K/(N-1)}(x_0-a)} \right )^{N-1}
\end{equation}
for all $ [x_0,  x_1] \subset [a, b] \subset  \supp h$. Moreover, we have the following known characterization (c.f. \cite[(2.10)]{CS18}):
\begin{lemma}\label{lemma:mcp} 
A  density $h$ is in $ \mathcal {MCP}^1_{K, N, D}$ if and only if $D \leq D_{K,N}$ and:
\begin{equation}\label{eq1.1}
\left (\frac {s_{K/(N-1)}(D-x_1)}{s_{K/(N-1)}(D-x_0)} \right )^{N-1} \leq \frac{h(x_1)}{h(x_0)} \leq \left (\frac {s_{K/(N-1)}(x_1)}{s_{K/(N-1)}(x_0)} \right )^{N-1}~~~\forall ~0\leq x_0 \leq x_1 \leq D.
\end{equation}
\end{lemma}
\begin{proof}
Immediate from (\ref{eq1}) after checking that for $0 \leq x_0 \leq x_1 \leq D$ the function
\[
a  \mapsto \frac {s_{K/(N-1)}(x_1-a)}{s_{K/(N-1)}(x_0-a)}
\]
 is non-decreasing on $[0, x_0]$, and the function
\[
b  \mapsto \frac {s_{K/(N-1)}(b-x_1)}{s_{K/(N-1)}(b-x_0)}
\] 
is non-decreasing on $[x_1, D]$. 
\end{proof}

This gives rise to the following definition:
\begin{definition}
Given $D \leq D_{K,N}$, the model $\mathcal{MCP}^1_{K,N,D}$ Poincar\'e density $h_{K,N,D}$ is defined by:
\[
h_{K,N,D} (x) := \left\{\begin{array}{lll}  
s^{N-1}_{K/(N-1)}(D-x) & \text{if}~~ x \in [0,D/2]  \\
s^{N-1}_{K/(N-1)}(x) & \text{if}~~ x \in [D/2,D] \\
\end{array}
\right .    .
\]
\end{definition}
\begin{remark}
Note that indeed $h_{K,N,D} \in \mathcal{MCP}^1_{K,N,D}$; this follows from (\ref{eq1.1}) and the fact that the function
\[
(0,D] \ni x \mapsto \frac{s_{K/(N-1)}(D-x)}{s_{K/(N-1)}(x)} 
\]
is decreasing, as verified in \cite[Lemma 3.3]{CS18}. 
$h_{K,N,D}$ is precisely the ``middle" model density (corresponding to $a = D/2$) from the family of isoperimetric $\mathcal{MCP}^1_{K,N,D}$ model densities $h^a_{K,N,D}$ considered by Cavalletti and Santarcangelo in \cite{CS18}. 
\end{remark}

We immediately deduce from (\ref{eq1.1}) (c.f. \cite[Lemma A.9]{CavallettiEMilman-LocalToGlobal}):
\begin{corollary} \label{cor:logLipschitz}
If $h \in \mathcal {MCP}^1_{K, N, D}$, then at every point $x \in [0,D]$ where $h$ if differentiable:
\[
-(\log s^{N-1}_{K/(N-1)})'(D - x) \leq (\log h)'(x) \leq (\log s^{N-1}_{K/(N-1)})'(x) . 
\]
In particular:
\[
(\log h)'(x) \left\{\begin{array}{lll}  
\geq (\log h_{K,N,D})'(x) & \text{if}~~ x \in [0,D/2)  \\
\leq (\log h_{K,N,D})'(x) & \text{if}~~ x \in (D/2,D] \\
\end{array}
\right . .
\]
\end{corollary}

\subsection{One dimensional Poincar\'e inequality} 

Given a density $h \in L^1_{loc}(\R,\mathcal{L}^1)$, denote its associated Poincar\'e constant on an interval $I \subset \R$ by \[
\lambda[h,I] := \lambda_{(I , |\cdot| , h \mathcal{L}^1)} = \inf \left \{\frac {\int_{I}   |f'|^2 \, h \, dx}{\int_{I} |f|^2  \, h \, dx}: f \in \Lip_{loc}(I) , \int_{I} f \, h \, dx =0,  0 < \int_I |f|^2 \, h \, dx < \infty  \right \} .
\]
We abbreviate $\lambda[h] := \lambda[h, \R]$. By a classical variational argument (cf. \cite[Proposition 4.5.4]{BGL-Book} or \cite[Theorem 4.2]{BobkovGotze-HardyInqs}),
 the Poincar\'e constant coincides with the first non-zero Neumann eigenvalue for all (say) piecewise smooth densities $h$ on $I$:
\[
\lambda[h,I] = \Lambda^{N,N}(h,I) . 
\]

In addition, the following simple perturbation lemma is well-known (see e.g. \cite[Proposition 5.5]{Ledoux-book}):
\begin{lemma} \label{lem:HS}
Given two positive densities $h_1,h_2$ on an interval $I \subset \R$, denote:
\[
\osc(h_1,h_2,I) := \esssup_{x \in I} \frac{h_2(x)}{h_1(x)} \cdot \esssup_{x \in I} \frac{h_1(x)}{h_2(x)} . 
\]
Then:
\[
\frac{1}{\osc(h_1,h_2,I)} \lambda[h_1,I] \leq \lambda[h_2,I] \leq  \osc(h_1,h_2,I) \lambda[h_1,I] . 
\]
\end{lemma}

We are now ready to establish the following sharp estimate:
\begin{proposition} \label{prop:diameter-D}
Let $h$ be a $\MCP(K,N)$ density with $\diam(\supp h) = D \in (0,\infty)$, $K \in \R$ and $N \in (1,\infty)$. Then the following sharp estimate holds:
\[
\lambda[h] \geq \lambda[h_{K,N,D}] .
\]
\end{proposition} 
\begin{proof}
As $(\supp h,|\cdot|)$ is geodesic, it must be a compact interval; by translation invariance, we may assume that $\supp h = [0,D]$. 
If $D = D_{K,N}$ it follows immediately from (\ref{eq1.1}) that necessarily $h(x) = c \cdot s^{N-1}_{K/(N-1)}(x)$ for some $c > 0$, 
and so $\lambda[h] = \lambda[s_{K/(N-1)}] = \lambda[h_{K,N,D_{K,N}}]$ and there is nothing further to prove; consequently, we may assume that $D < D_{K,N}$. 
We now reduce to the case that $h$ is smooth and positive on its support.  While this follows from a fairly simple approximation argument, we take the time to sketch its proof, as one may find various erroneous approximation arguments in the literature (in the $\CD(K,N)$ setting).

It is known that the $\MCP(K,N)$ density $h$ is bounded above on $[0,D]$, positive on $(0,D)$, and that $\log h$ is locally Lipschitz on $(0,D)$ (see \cite[Lemmas A.8 and A.9]{CavallettiEMilman-LocalToGlobal} which were stated for $\CD(K,N)$ densities, but the proof only uses the defining property of $\MCP(K,N)$ densities). 
 Let $\varphi$ denote a smooth compactly supported non-negative function on $\R$ supported on $[-1,1]$ which integrates to $1$, and denote by $\varphi_\eps(x) := \frac{1}{\eps} \varphi(x/\eps)$, $\eps > 0$, the corresponding family of mollifiers. By definition, the family of $\MCP(K,N)$ densities \emph{having fixed support $I$} is a convex cone (note that this is totally false if the supports do not coincide).  Since the restriction of $h$ onto any non-empty sub-interval of $[0,D]$ is itself an $\MCP(K,N)$ density, it follows that the convolution $h_{\eps} = h \ast \varphi_\eps$ is an $\MCP(K,N)$ density when restricted to $[\eps,D-\eps]$ (but possibly not on $[0,D]$). It is a standard fact that $h_\eps$ is smooth and that $h_\eps \rightarrow h$ uniformly on $[\delta , D -  \delta]$ as $\eps \rightarrow 0+$ for any fixed $\delta > 0$. As $h$ is strictly positive on $[\delta,D-\delta]$, if follows that $h_\eps/ h \rightarrow 1$ uniformly on $[\delta, D - \delta]$, and hence by Lemma \ref{lem:HS} we deduce that:
\begin{equation} \label{eq:eps-convergence}
\lim_{\eps \rightarrow 0+} \lambda[h_\eps,[\delta,D-\delta]] = \lambda[h, [\delta , D - \delta]] ,
\end{equation}
for any fixed $\delta > 0$. 

Now consider the model Poincar\'e density $h_{K,N,D- 2 \delta}$, which we henceforth translate so that it is supported on $[\delta , D - \delta]$. The eigenfunction associated to the first non-zero Neumann eigenvalue $\Lambda^{N,N}(h_{K,N,D-2 \delta},[\delta, D - \delta])$ is strictly monotone and necessarily has its unique zero at $D/2$, by the symmetry of $h_{K,N,D-2\delta}$ around this point and the fact that the eigenvalue is simple. Since by Corollary \ref{cor:logLipschitz}
 \[
(\log h_\eps)'(x) \left\{\begin{array}{lll}  
\geq (\log h_{K,N,D-2 \delta})'(x) & \text{if}~~ x \in [\delta ,D/2)  \\
\leq (\log h_{K,N,D-2 \delta })'(x) & \text{if}~~ x \in (D/2,D-\delta] \\
\end{array}
\right . ,
\]
it follows from Lemma \ref{lem:ODE} (ODE comparison principle) that:
\[
\lambda[h_\eps,[\delta,D-\delta]] = \Lambda^{N,N}(h_\eps, [\delta,D-\delta]) \geq \Lambda^{N,N}(h_{K,N,D-2 \delta},[\delta, D - \delta])  = \lambda[h_{K,N,D-2 \delta},[\delta, D - \delta]] .
\]
Taking the limit as $\eps \rightarrow 0+$, (\ref{eq:eps-convergence}) implies that:
\begin{equation} \label{eq:delta-convergence}
\lambda[h, [\delta , D - \delta]]  \geq \lambda[h_{K,N,D-2 \delta},[\delta, D - \delta]] .
\end{equation}

Finally, observe that Lemma \ref{lem:HS} implies that:
\[
\lambda[h_{K,N,D-2 \delta},[\delta, D - \delta]] \geq c_{\delta}\lambda[h_{K,N,D},[\delta, D - \delta]] ,
\]
with $\lim_{\delta \rightarrow 0+} c_{\delta} = 1$ (recall that $D < D_{K,N}$ so that the density $h_{K,N,D}$ is positively bounded below on $[0,D]$). 
It remains to invoke e.g. \cite[Theorem 5.2.4]{Calderon18}, where it is shown that for any $f \in L^1([0,D], \mathcal{L}^1)$:
\[
\lim_{\delta \rightarrow 0+} \lambda[f , [\delta,D-\delta]] = \lambda[f,[0,D]] 
\]
(in fact, we just need the upper semi-continuity, which is particularly simple). 
Applying this to (\ref{eq:delta-convergence}), it follows that:
\[
\lambda[h,[0,D]] = \lim_{\delta \rightarrow 0+} \lambda[h, [\delta , D - \delta]]  \geq \lim_{\delta \rightarrow 0+} c_{\delta} \lambda[h_{K,N,D},[\delta, D - \delta]] = \lambda[h_{K,N,D} , [0,D]],
\]
as asserted. 
\end{proof}

If we only have an upper bound on the diameter of the support of $h$, we deduce:

\begin{corollary} \label{cor:diameter-at-most-D}
Let $h$ be a $\MCP(K,N)$ density with $\diam(\supp h) \leq D \in (0,\infty)$, $K \in \R$ and $N \in (1,\infty)$. Then the following sharp estimate holds:
\[
\lambda[h] \geq \left\{\begin{array}{lll}  
\lambda[h_{K,N,D}] & \text{if}~~ K \leq 0 \\
\inf_{D' \in (0,\min(D,D_{K,N})]} \lambda[h_{K,N,D'}] & \text{if}~~ K > 0 \\
\end{array}
\right . .
\]
\end{corollary}
\begin{proof}
Let $D' = \diam(\supp h) \in (0,D]$; by Bonnet--Myers (\ref{eq:BM}), we also know that $D' \leq D_{K,N}$. 
By Proposition \ref{prop:diameter-D} we have:
\[
\lambda[h] \geq \lambda[h_{K,N,D'}] ,
\]
which yields the assertion when $K > 0$. 

When $K \leq 0$, it remains to establish that:
\begin{equation} \label{eq:monotone}
(0,\infty) \ni D' \mapsto \lambda[h_{K,N,D'}] \text{ is strictly decreasing, }
\end{equation}
thereby concluding the proof. In fact, a stronger property holds, namely:
\begin{lemma} \label{lem:monotone}
The mapping:
\[
(0,D_{K,N}] \ni D' \mapsto (D')^2 \lambda[h_{K,N,D'}] 
\]
is non-increasing if $K \leq 0$ and non-decreasing if $K \geq 0$. In particular, it is constant if $K=0$. 
\end{lemma}
\begin{proof}
Let $0 < D' \leq D \leq D_{K,N}$ with $D < \infty$. Assume $K \leq 0$, and consider the scaled density $h_{K,N,D'}(\frac{D'}{D} x)$ which is supported on $[0,D]$ and satisfies $\MCP((\frac{D'}{D})^2 K ,N)$;  since $(\frac{D'}{D})^2 K \geq K$ when $K \leq 0$, it also satisfies $\MCP(K ,N)$, and so by Proposition \ref{prop:diameter-D} and scaling of the Poincar\'e constant we deduce the claim:
\[
\brac{\frac{D'}{D}}^2 \lambda[h_{K,N,D'}] = \lambda[h_{K,N,D'}((D'/D) x)] \geq \lambda[h_{K,N,D}] . 
\]
The case $K \geq 0$ is treated analogously, exchanging the roles of $D'$ and $D$. 
\end{proof}
\end{proof}

The difference between the cases $K \leq 0$ and $K > 0$ was already observed in \cite{CS18} in the isoperimetric context. In the next section, we will verify that it is not an artifact of the proof; in particular, the monotonicity property (\ref{eq:monotone}) is \textbf{false} when $K> 0$ in the relevant range $D' \in (0,D_{K,N}]$. It is an interesting question whether the function
\[
(0,D_{K,N}] \ni D' \mapsto \lambda[h_{K,N,D'}] 
\]
is at least unimodal when $K >0$, and if so, to determine where its unique minimum is attained. We provide some partial answers in the next section.

\section{Estimating $\lambda[h_{K,N,D}]$} \label{sec:estimates}

In this section, we study the quantitative dependence of $\lambda[h_{K,N,D}]$ on the parameters $K$, $N$ and $D$. Note that $h_{K a^2, N, D/a}(x) = \frac{1}{a} h_{K , N , D}(a x)$ for any $a > 0$, and so scaling of the Poincar\'e constant implies: 
\[
\lambda[h_{K a^2,N, D/a}] = a^2 \lambda[h_{K,N,D}] . 
\]
Consequently, it is only necessary to treat the cases $K=0$, $K=-1$ and $K=1$, but as this comes at no extra cost, we will analyze the general cases $K <0$ and $K> 0$ below. 
In addition, since $\MCP(K,N)$ implies $\MCP(K',N')$ for any $K' \leq K$ and $N' \geq N$, and since $h_{K,N,D}$ is an $\MCP(K,N)$ density supported on an interval of length $D$, it follows by Proposition \ref{prop:diameter-D} that:
\begin{equation} \label{eq:trivial}
\lambda[h_{K,N,D}] \geq \lambda[h_{K',N',D}] . 
\end{equation}

\medskip

To obtain more meaningful estimates, we will use 
the following classical result, first derived by Kac and Krein \cite{KacKreinVibratingString}, later by Artola, Talenti and Tomaselli (separately and independently), and generalized by Muckenhoupt, thereby bearing his name (see \cite{MuckenhouptHardyInq} and the references therein). For simplicity, we only state the version we require here (see e.g. \cite[Theorem 1.2]{BobkovGotze-HardyInqs}). 

\begin{proposition}[Muckenhoupt's criterion] \label{prop:Muck}
For any smooth positive density $\Psi$ on a compact interval $I = [a,b] \subset \R$, denote: \[
A[\Psi,I] := \sup_{x \in [a,b]} \int_a^x \frac{dt}{\Psi(t)} \int_x^b \Psi(t) dt . 
\]
Then:
\[
A[\Psi,I] \leq \frac{1}{\Lambda^{D,N}(\Psi,I)} \leq 4 A[\Psi,I] . 
\]
\end{proposition}

As explained in the previous section,
\[
\lambda[h_{K,N,D}] = \Lambda^{N,N}(h_{K,N,D},[0,D]) = \Lambda^{D,N}(h_{K,N,D},[D/2,D]) ,
\]
and so Proposition \ref{prop:Muck} provides us with a way to estimate $\lambda[h_{K,N,D}]$ quite well. 

\subsection{Case $K=0$}

\begin{lemma} \label{lem:zero}
For all $N \in (1,\infty)$ and $D \in (0,\infty)$:
\[
\pi^2 N^2 2^{-(N-1)} \geq D^2 \lambda[h_{0,N,D}] \geq \frac{1}{4} N^2 2^{-(N-1)} . 
\]
\end{lemma}
\noindent
Note that $2^{-(N-1)} = \sigma_{0,N}^{(1/2)}(D)^{N-1}$. 
\begin{proof}
Recall that $D^2 \lambda[h_{0,N,D}]$ is independent of $D$ by Lemma \ref{lem:monotone}, so we may assume $D=1$. Assume first that $N \geq 4$. 
Our task it to evaluate:
\begin{align}
 \label{eq:AN} A_N := A[h_{0,N,1},[1/2,1]] & = \sup_{x \in [1/2,1]} \int_{\frac{1}{2}}^x \frac{dt}{t^{N-1}} \int_{x}^1 t^{N-1} dt \\
\nonumber & = \sup_{x \in [1/2,1]} \frac{ 2^{N-2} - x^{2-N}}{N-2} \frac{1 - x^N}{N}  .
\end{align}
As $N \geq 4$, we trivially upper bound this by:
\begin{equation} \label{eq:AN-bound}
A_N \leq \frac{1}{4} \frac{1}{N (N-2)} 2^{N} \leq \frac{1}{4} \frac{2^{N-1}}{N^2} ,
\end{equation}
and the asserted lower bound follows by Proposition \ref{prop:Muck}. On the other hand, using $x^{2-N} \leq x^{-N}$ (as $x \in [1/2,1]$), we have:
\[
A_N \geq \frac{1}{N (N-2)} \sup_{x \in [1/2,1]} (2^{N-2} - x^{-N})(1 - x^N) = 
\frac{1}{N (N-2)} (2^{N-2} + 1 - 2 \cdot 2^{N/2-1}) .
\]
As $N \geq 4$, it is easy to check that this implies:
\[
A_N \geq \frac{1}{16} \frac{2^N}{N^2} \geq \frac{1}{\pi^2} \frac{2^{N-1}}{N^2}  ,
\]
and the asserted upper bound follows by Proposition \ref{prop:Muck}.

When $N \in (1,4)$, we can simply invoke Lemma \ref{lem:HS} to compare $h_{0,N,1}$ with the constant density $1$. Since $\lambda[1,[0,1]] = \pi^2$, we obtain the lower bound below:
\[
1 \geq \frac{\lambda[h_{0,N,1}]}{\pi^2} \geq 2^{-(N-1)} ;
\]
the upper bound follows by Proposition \ref{prop:diameter-D} since $1$ is itself an $\MCP(0,N)$ density on $[0,1]$. 
It is immediate to check that $2^{-(N-1)} \pi^2 \geq \frac{1}{4} N^2 2^{-(N-1)}$ and that $\pi^2 \leq 2 \pi^2 N^2 2^{-N}$ when $N \in (1,4)$, thereby concluding the proof. 
\end{proof}

\subsection{Case $K<0$}

A similar argument applies to the case $K < 0$. For brevity, we only supply the lower bound:
\begin{lemma}
For all $K < 0$, $N \in (1,\infty)$ and $D \in (0,\infty)$:
\[
\lambda[h_{K,N,D}] \geq \frac{1}{4} \max\brac{|K|(N-1), \frac{N^2}{D^2}} \sigma_{K,N}^{(1/2)}(D)^{N-1} . \]
\end{lemma}
\noindent
Note that when $K \rightarrow 0-$, the limiting lower bound precisely coincides with the one from the previous lemma.
\begin{proof}
By scaling, we have for any $T > 0$:
\[
\lambda[h_{K,N,D}] = \frac{1}{T^2} \lambda[h_{T^2 K , N , D/T}] ,
\]
and therefore, denoting $D' = \sqrt{\frac{-K}{N-1}} D$,
\[
\lambda[h_{K,N,D}] = \frac{-K}{N-1} \lambda[h_{-(N-1),N,D'}] . 
\]
Our task it to evaluate:
\[
B_{N,D'} := A[h_{-(N-1),N,D'},[D'/2,D']] = \sup_{x \in [D'/2,D']} \int_{\frac{D'}{2}}^x \frac{dt}{\sinh^{N-1}(t)} \int_{x}^{D'} \sinh^{N-1}(t) dt .
\]
Since $\sinh(t) / e^t$ is an increasing function on $\R_+$, we first evaluate:
\begin{align*}
B_{N,D'} &  \leq \brac{\frac{\sinh(D')}{\sinh (D'/2)}}^{N-1} \sup_{x \in [D'/2,D']} \int_{\frac{D'}{2}}^x e^{(\frac{D'}{2} - t)(N-1)} dt \int_x^{D'} e^{(t-D')(N-1)} dt \\
& \leq \brac{\frac{\sinh(D')}{\sinh (D'/2)}}^{N-1}  \min\brac{\frac{1}{(N-1)^2} , \brac{\frac{D'}{4}}^2} . 
\end{align*}
In addition, since $\sinh(t)/t$ is increasing on $\R_+$, we also obtain when $N \geq 4$:
\begin{align*}
B_{N,D'} &  \leq \brac{\frac{\sinh(D')}{\sinh (D'/2)}}^{N-1} \sup_{x \in [D'/2,D']} \int_{\frac{D'}{2}}^x \brac{\frac{D'/2}{t}}^{N-1} dt \int_x^{D'} \brac{\frac{t}{D'}}^{N-1} dt \\
& = \brac{\frac{\sinh(D')}{\sinh (D'/2)}}^{N-1} \frac{(D')^2}{2^{N-1}} A_N \leq \brac{\frac{\sinh(D')}{\sinh (D'/2)}}^{N-1} \frac{(D')^2}{N^2} , \end{align*}
where $A_N$ was defined in (\ref{eq:AN}) and we employed (\ref{eq:AN-bound}) in the last inequality. Combining our estimates and recalling the definition of $D'$, we obtain for all $N \in (1,\infty)$:
\[
B_{N,D'} \leq \frac{1}{\sigma_{K,N}^{(1/2)}(D)^{N-1}} \min\brac{\frac{1}{(N-1)^2} , \frac{|K|}{N-1} \frac{D^2}{N^2} }. 
\]
Applying Proposition \ref{prop:Muck}, the assertion follows. 
\end{proof}

\subsection{Case $K > 0$}

\begin{lemma}
For all $K > 0$, $N \in (1,\infty)$ and $0 < D' \leq D \leq D_{K,N}$:
\[
(D')^2 \lambda[h_{K,N,D'}] \leq 
D^2 \lambda[h_{K,N,D}] \leq
(D')^2 \lambda[h_{K,N,D'}] \brac{\frac{\sigma_{K,N}^{(1/2)}(D)}{\sigma_{K,N}^{(1/2)}(D')}}^{N-1} . 
\]
\end{lemma}
\begin{proof}
The first inequality was already established in Lemma \ref{lem:monotone}. For the second inequality, consider the rescaled density $h_{K,N,D'}(\frac{D'}{D} t)$ on $[0,D]$, and compare its Poincar\'e constant to that of $h_{K,N,D}(t)$ using Lemma \ref{lem:HS}. By scaling and symmetry:
\[
\lambda[h_{K,N,D}] \leq \brac{\frac{D'}{D}}^2  \lambda[h_{K,N,D'}] \osc(h_{K,N,D}(t),h_{K,N,D'}((D'/D) t),[D/2,D]) , 
\]
where recall:
\[
\osc(h_{K,N,D}(t),h_{K,N,D'}((D'/D) t ),[D/2,D])  = \brac{\frac{\max_{t \in [D/2,D]} \frac{s_{K/(N-1)}(t)}{s_{K/(N-1)}((D'/D)t)}}{\min_{t \in [D/2,D]} \frac{s_{K/(N-1)}(t)}{s_{K/(N-1)}((D'/D)t)}}}^{N-1} .
\]
By directly calculating the derivative, 
it is straightforward to check that the function
$t \mapsto \frac{s_{K/(N-1)}(t)}{s_{K/(N-1)}((D'/D)t)}$ is non-increasing on $(0,D_{K,N}]$ when $0 < D' \leq D$; alternatively, one can use the fact that $(-\infty, \log(\pi))  \ni x \mapsto \log \sin \exp(x)$ is concave, which implies that $x \mapsto \log \sin \exp(x) - \log \sin(a \exp(x))$ is non-increasing for $a \in (0,1]$.
Consequently, the above maximum and minimum are attained at $t=D/2$ and $t=D$, respectively, and the assertion follows. 
\end{proof}

We will exploit the fact that we can recognize the limit of $(D')^2 \lambda[h_{K,N,D'}]$ at both endpoints of the interval $D ' \in [0,D_{K,N}]$. 

\begin{corollary} \label{cor:positive-zero-compare}
\[
\lambda[h_{0,N,D}] \leq \lambda[h_{K,N,D}] \leq \lambda[h_{0,N,D}] \brac{\frac{\sigma_{K,N}^{(1/2)}(D)}{\sigma_{0,N}^{(1/2)}(D)}}^{N-1} = \lambda[h_{0,N,D}] \brac{\frac{2 \sin(\sqrt{\frac{K}{N-1}} \frac{D}{2})}{ \sin(\sqrt{\frac{K}{N-1}} D)}}^{N-1} .
\]
\end{corollary}
\begin{proof}
Taking the limit $D' \rightarrow 0+$ in the previous lemma, it is clear that the limit of $(D')^2 \lambda[h_{K,N,D'}]$ is independent of $K$ (say by Lemma \ref{lem:HS}), as the curvature effect is unnoticeable at infinitesimal scales. Consequently, we have:
\begin{align*}
D^2 \lambda[h_{K,N,D}]  & \leq \lim_{D' \rightarrow 0+} (D')^2 \lambda[h_{K,N,D'}] \brac{\frac{\sigma_{K,N}^{(1/2)}(D)}{\sigma_{K,N}^{(1/2)}(D')}}^{N-1} \\
& = \lim_{D' \rightarrow 0+} (D')^2 \lambda[h_{0,N,D'}] \brac{\frac{\sigma_{K,N}^{(1/2)}(D)}{\sigma_{0,N}^{(1/2)}(D')}}^{N-1} \\
& = D^2 \lambda[h_{0,N,D}] \brac{\frac{\sigma_{K,N}^{(1/2)}(D)}{\sigma_{0,N}^{(1/2)}(D)}}^{N-1} ,
\end{align*}
where the last equality follows since all relevant expressions are scale invariant when $K=0$. This establishes the second inequality of the assertion; the first follows identically, or simply by (\ref{eq:trivial}) since $K > 0$. 
\end{proof}

As for the other endpoint, note that $h_{K,N,D_{K,N}}$ is simply the density $\sin(\sqrt{\frac{K}{N-1}} t)^{N-1}$ on $[0,D_{K,N}]$. In this special case, the model $\MCP(K,N)$ Poincar\'e density coincides with its $\CD(K,N)$ counterpart, which corresponds to the density obtained from pushing forward the uniform measure on an $N$-dimensional sphere having Ricci curvature equal to $K$ via the radial map $x \mapsto \d(x,x_0)$. Consequently, we know that:
\begin{lemma} \label{lem:sphere}
\[
\lambda[h_{K,N,D_{K,N}}] = \Lambda^{N,N}(h_{K,N,D_{K,N}} , [0,D_{K,N}]) = \frac{N}{N-1} K . 
\]
\end{lemma}
\begin{proof}
Observe that $u(t) = \cos(\sqrt{\frac{K}{N-1}} t)$ is a monotone function on $[0,D_{K,N}]$ satisfying Neumann boundary conditions there and that:
\[
-\Delta^{N,N}_{h_{K,N,D_{K,N}}} u = \frac{N}{N-1} K u . 
\]
Hence $u$ must be the eigenfunction corresponding to the first non-zero Neumann eigenvalue. 
\end{proof}

From the previous discussion we can already conclude that when $K > 0$:
\begin{equation} \label{eq:false}
(0,D_{K,N}] \ni D \mapsto \lambda[h_{K,N,D}] \text{ is \textbf{not} non-increasing.}
\end{equation}
This is in stark contrast to the case $K \leq 0$, when the above function is strictly decreasing by (\ref{eq:monotone}), and explains why in the formulations of Theorem \ref{thm:main} and Corollary \ref{cor:diameter-at-most-D} we really need to take an infimum over $D' \in (0,\min(D,D_{K,N})]$ when $K > 0$.

Let us only show (\ref{eq:false}) for $N \geq 13$. Indeed, by Corollary \ref{cor:positive-zero-compare} and Lemma \ref{lem:zero}, we know that:
\[
\lambda[h_{K,N,D}] \leq \frac{\pi^2 N^2}{D^2} \brac{\frac{\sin(\sqrt{\frac{K}{N-1}} \frac{D}{2})}{ \sin(\sqrt{\frac{K}{N-1}} D)}}^{N-1} . 
\]
Setting $D = \alpha D_{K,N}$, $\alpha \in (0,1)$, and recalling Lemma \ref{lem:sphere}, this is equal to:
\[
= \lambda[h_{K,N,D_{K,N}}] \frac{N}{\alpha^2} \brac{\frac{\sin(\frac{\alpha}{2} \pi)}{ \sin(\alpha \pi )}}^{N-1} . 
\]
When $N \geq 13$, it is immediate to verify that for $\alpha = 1/2$, the term on the right is strictly smaller than $1$, and we deduce $\lambda[h_{K,N,D_{K,N}/2}] < \lambda[h_{K,N,D_{K,N}}]$. 

It is also possible to show using Hadamard's formula that the derivative of $D \mapsto \lambda[h_{K,N,D}]$ is strictly positive at $D = D_{K,N}$, thereby verifying (\ref{eq:false}) for all $N > 1$; we omit the details.

\section{Proof of Main Theorem} \label{sec:proof}

We are now ready to prove our Main Theorem \ref{thm:main}. 

\begin{proof}[Proof of Theorem \ref{thm:main}]
Given a Lipschitz function $f$ on $(X,\d)$ with $\int_\Omega f \mm = 0$, set $g = f 1_{\Omega}$. As $(\supp(\mm),\d)$ is proper and $\mm$ is locally finite, the integrability assumption $\int_X |g(x)| d(x,x_0) \mm(dx) < \infty$ is clearly satisfied, and we may apply the Localization Theorem \ref{thm:localization}. It follows that there exists an $\mm$-measurable subset $\mathsf T \subset X$ and a family $\{X_{q}\}_{q \in Q} \subset X$, such that: 
\begin{enumerate}
\item There exists a disintegration of $\mm\restr{\mathsf{T}}$ on $\{X_{q}\}_{q \in Q}$:
\[
\mm\restr{\mathsf{T}} = \int_{Q} \mm_{q} \, \qq(dq)  ~,~ \qq(Q) = 1 . 
\]  
\item For $\qq$-a.e. $q \in Q$, $X_q$ is a closed geodesic in $(X,\d)$.
\item For $\qq$-a.e. $q \in Q$, $\mm_q$ is a Radon measure supported on $X_q$ with $\mm_q \ll  \mathcal H^1 \restr{X_q}$.
\item For $\qq$-a.e. $q \in Q$, the metric measure space $(X_{q}, \d,\mm_{q})$ verifies $\MCP(K,N)$.
\item For $\qq$-a.e. $q \in Q$, $\int_{\Omega} f \mm_q = \int_X g \mm_q = 0$, and $f \equiv 0$  $\mm$-a.e. on $\Omega \setminus \mathsf T$. 
\end{enumerate}

Since $\supp(g \mm) \subset \Omega$, we know that $\diam(\supp(g \mm)) \leq D$. 
Let $q \in Q$ be such that all of the above properties hold, and denote:
\[
L_q := \conv_{X_q}(\supp(g \mm) \cap X_q) \; ;
\]
although this is not important, we point out that we take the geodesic hull inside the metric space $(X_{q}, \d)$ which is isometric to a closed subinterval of $(\R,|\cdot|)$. It follows that $\diam(L_q) \leq D$, and we have:
\begin{equation} \label{eq:Lq}
\supp(g \mm) \cap X_q \subset L_q \subset \conv(\supp(g \mm)) \cap X_q  .
\end{equation}
Since $\mm|_T(\{ g \neq 0\} \setminus \supp(g \mm)) = 0$, the above disintegration and Fubini's theorem imply that for $\qq$-a.e. $q \in Q$, $g \equiv 0$ $\mm_q$-a.e. on $X \setminus \supp(g \mm)$ and in particular on $X_q \setminus L_q$, and hence:
\begin{enumerate}
\setcounter{enumi}{5}
\item For $\qq$-a.e. $q \in Q$, $f \equiv 0$ $\mm_q$-a.e. on $X_q \cap \Omega \setminus (L_q \cap \Omega)$.
\end{enumerate}
We therefore add this requirement from $q$ to our previous requirements, as they all hold for $\qq$-a.e. $q \in Q$. 

 Since the $\MCP(K,N)$ condition is closed under restrictions onto geodesically convex subsets, it follows that $(L_q , \d, \mm_q|_{L_q})$ verifies $\MCP(K,N)$; however, since $\Omega$ was not assumed to be geodesically convex, note that $(L_q \cap \Omega, \d, \mm_q|_{L_q \cap \Omega})$ \textbf{may not} satisfy $\MCP(K,N)$. Nevertheless, we claim that:
 \begin{equation} \label{eq:goal}
\int_{L_q \cap \Omega} f^2 \mm_q \leq \frac{1}{\lambda_{{\mathcal {MCP}}_{K,N,D}}} \int_{L_q} |\nabla_{L_q} f|^2 \mm_q \; ,
\end{equation}
where recall $\lambda_{{\mathcal {MCP}}_{K,N,D}}$ was defined in (\ref{eq:intro-optimum}), and $|\nabla_{L_q} f|$ is the local Lipschitz constant of $f$ on $(L_q,\d)$. 

 To see this, first note that by property (6):
\begin{equation} \label{eq:LqBallanced}
 \int_{L_q\cap \Omega} f \mm_q  = \int_{X_q \cap \Omega} f \mm_q = 0 \; ;
\end{equation}
however,  $\int_{L_q} f \mm_q$ may not vanish since $L_q$ could exit and reenter $\Omega$ if $\Omega$ is not geodesically convex. To establish (\ref{eq:goal}), we may assume that $\mm_q(L_q) > 0$, since otherwise there is nothing to prove. We know that the one-dimensional metric measure space $(L_q , \d, \mm_q|_{L_q})$ is isometric to $(I,|\cdot|,h_q \mathcal{L}^1)$ for some closed interval $I \subset \R$ with $\diam(I) \leq D$ and with $h_q$ an $\MCP(K,N)$ density (by Lemma \ref{lem:MCP-density}), and we identify these two representations. Applying Corollary \ref{cor:diameter-at-most-D} to the function $\bar f := f - \frac{1}{\mm_q(L_q)} \int_{L_q} f \mm_q$, since $\int_{L_q} \bar f \mm_q = 0$, we deduce:
\[
\int_{L_q} f^2 \mm_q - \frac{(\int_{L_q} f \mm_q)^2}{\mm_q(L_q)} = \int_{L_q} (\bar f)^2 \mm_q \leq \frac{1}{\lambda_{{\mathcal {MCP}}_{K,N,D}}} \int_{L_q} |\nabla_{L_q} f|^2 \mm_q .
\]
This immediately implies (\ref{eq:goal}) if $\mm_q(L_q \setminus \Omega) = 0$ by (\ref{eq:LqBallanced}), while otherwise, (\ref{eq:goal}) follows since:
\[
\int_{L_q \setminus \Omega} f^2 \mm_q \geq \frac{(\int_{L_q \setminus \Omega} f \mm_q)^2}{\mm_q(L_q \setminus \Omega)} \geq 
\frac{(\int_{L_q \setminus \Omega} f \mm_q)^2}{\mm_q(L_q)} = \frac{(\int_{L_q} f \mm_q)^2}{\mm_q(L_q)}  ,
\]
where we employed (\ref{eq:LqBallanced}) again  in the final transition. 

Recalling Property (6) and (\ref{eq:Lq}), (\ref{eq:goal}) implies:
\[
\int_{X_q \cap \Omega} f^2 \mm_q \leq \frac{1}{\lambda_{{\mathcal {MCP}}_{K,N,D}}} \int_{X_q \cap \conv(\supp(g \mm))} |\nabla_{L_q} f|^2 \mm_q . 
\]
Using $|\nabla_{L_q} f|\leq |\nabla_X f|$ and integrating this with respect to $\qq$, we deduce (after recalling that $f \equiv 0$ $\mm$-a.e. on $\Omega \setminus T$):
\[
\int_{\Omega} f^2 \mm = \int_{T \cap \Omega} f^2 \mm \leq \frac{1}{\lambda_{{\mathcal {MCP}}_{K,N,D}}} \int_{T \cap \conv(\supp(g \mm))} |\nabla_{X} f|^2 \mm .
\]
Since $\conv(\supp (g \mm)) \subset \conv(\Omega)$, this concludes the proof.
\end{proof}

\section{Comparison with prior results} \label{sec:conclude}

Before concluding, we mention some previously known related results on $\MCP(K,N)$ spaces. 

In \cite[Theorem 6.4 and Corollary 6.6]{SturmCD2}, Sturm obtained a certain Poincar\'e inequality on geodesic balls of an $\MCP(K,N)$ space ($N \in (1,\infty)$) under the assumption that the function
\[
w_N(x) := \lim_{\eps \rightarrow 0+} \frac{m(B_\eps(x))}{\eps^N} 
\]
is locally bounded; the inequality reads:
\[
\int_{B_r(x_0)} f \mm = 0 \Rightarrow \lambda^w_{K,N}(r) \frac{\mm(B_r(x_0))}{r^N} \int_{B_r(x_0)} f^2 \mm \leq \int_{B_{3r}(x_0)} w_N |\nabla_X f|^2 \mm ,
\]
with:
\[
\lambda^w_{K,N}(r) = \frac{1}{(2r)^2} \cdot \left\{\begin{array}{lll}  
\frac{2+N}{N 2^{N}} & \text{if}~~ K \geq 0 \\
\frac{2+N}{N 2^{N}} \brac{\frac{2 r}{s_{K/(N-1)}(2 r)}}^{N-1}  & \text{if}~~ K < 0 \\
\end{array}
\right . .
\]
While the $\frac{\mm(B_r(x_0))}{r^N}$ term may be bounded from below on any compact set by employing Bishop--Gromov volume comparison (valid on $\MCP(K,N)$ spaces -- see \cite{SturmCD2,Ohta-MCP}), it is not clear how to control $w_N$, or how to offset it using the $\frac{\mm(B_r(x_0))}{r^N}$ term (Bishop--Gromov goes in the wrong direction here). 
All in all, we do not see how to obtain an explicit quantitative expression for the Poincar\'e constant on balls from this approach.  

\medskip

In \cite{VonRenesse-L1PoincareOnMCP}, M.~von Renesse obtained the following $L^1$-Poincar\'e inequality on geodesic balls of an $\MCP(K,N)$ space ($N \in (1,\infty)$) under a certain non-branching assumption, which in particular holds if for $\mm$-a.e. point $x$, the cut-locus of $x$ has zero $\mm$-measure. It was shown in \cite[Remark 7.5]{CavallettiEMilman-LocalToGlobal} that the latter property holds on essentially non-branching $\MCP(K,N)$ spaces. Von Renesse's inequality reads:
\[
2 r \lambda^1_{K,N}(r) \frac{1}{\mm(B_r(x_0))} \int_{B_r(x_0)} \int_{B_r(x_0)} \frac{|f(x) - f(y)|}{\d(x,y)} \mm(dx) \mm(dy) \leq \int_{B_{2r}(x_0)} |\nabla_X f| \mm \; ;
\]
in particular, it implies (using $\d(x,y) \leq 2r$ and Jensen's inequality) that:
\[
\int_{B_r(x_0)} f \mm = 0 \Rightarrow  \lambda^1_{K,N}(r) \int_{B_r(x_0)} |f| \mm \leq \int_{B_{2r}(x_0)} |\nabla_X f| \mm \;,
\]
with:
\[
\lambda^1_{K,N}(r) = \frac{1}{2r} \inf_{t \in [1/2,1] , \theta \in (0,2r]} t \; \sigma_{K,N}^{(t)}(\theta)^{N-1} =  \frac{1}{2r} \cdot \left\{\begin{array}{lll}  
\frac{1}{2} \frac{1}{2^{N-1}}  & \text{if}~~ K \geq 0 \\
\frac{1}{2} \sigma_{K,N}^{(1/2)}(2r)^{N-1}  & \text{if}~~ K < 0 \\
\end{array}
\right .  . 
\]

We do not know how to quantitatively compare between the above $L^1$-Poincar\'e  inequality and our $L^2$-Poincar\'e one; it is always possible to pass from an $L^1$ tight version to an $L^2$ tight one (when $B_{2r}$ on the right-hand-side is replaced by $B_r$) by applying it to $f = g^2 \text{sgn}(g)$ and using Cauchy--Schwarz, but we do not know how to do this for the above non-tight version. Note that by the results of \cite{SobolevMetPoincare}, a non-tight Poincar\'e inequality may always be tightened on any geodesic space, but this results in loss of explicit constants. Still, it might be interesting to compare the above explicit expression for $2r \lambda^1_{K,N}(r)$ with our estimates on $(2r)^2 \lambda_{\mathcal{MCP}_{K,N,2r}}$ from Section \ref{sec:estimates} (with this scaling, both are unit-free). Besides the fact that our estimates apply to any $\Omega$ with $\diam(\Omega) \leq 2r$, two other notable differences are that our estimates improve when $K > 0$, and that we have an additional advantageous $N^2$ term when $K \leq 0$. In any case, as explained in the Introduction, our Poincar\'e constant $\lambda_{\mathcal{MCP}_{K,N,D}}$ is best possible. 

\medskip

We also mention a recent result of Eriksson-Bique \cite[Theorem 1.3]{Eriksson-Poincare}, who established a $(1,p)$-local-Poincar\'e inequality for $p > N+1$ on $\MCP(K,N)$ spaces, without any non-branching assumptions whatsoever (when $K \geq 0$ he also obtained a global version). 

\medskip

Lastly, it is worthwhile to mention the results of Yang and Lian \cite{YangLian-WeightedPoincare}, who obtained very precise (and in some cases optimal) constants for weighted Poincar\'e inequalities on geodesic balls $B_r$ in the Heisenberg group as well as other Carnot groups, with \emph{Dirichlet boundary conditions} (when the corresponding test functions are compactly supported inside $B_r$). The optimal constant for the global Sobolev inequality on $\H^n$ was discovered by Jerison and Lee \cite{JerisonLee-ExtremalsForSobolevOnHeisenberg}.

\def\cprime{$'$} \def\cprime{$'$} \def\textasciitilde{$\sim$}

\end{document}